\documentclass{amsart}

\usepackage{amsmath,amssymb,amsthm}
\usepackage{url}
\usepackage{hyperref}
\usepackage{enumerate}
\usepackage{graphicx}
\usepackage{booktabs}
\usepackage{subcaption}

\theoremstyle{definition}
\newtheorem{definition}{Definition}

\theoremstyle{plain}
\newtheorem{theorem}{Theorem}
\newtheorem{lemma}{Lemma}

\newtheorem{corollary}{Corollary}

\newcommand{\F}{\mathcal{F}}

\title{Cartesian lattice counting by the vertical 2-sum}

\begin{document}

\author{Jukka Kohonen}
\address{Department of Mathematics and Systems Analysis\\ Aalto University, Finland}
\email{jukka.kohonen@iki.fi}

\begin{abstract}
  A~vertical 2-sum of a two-coatom lattice~$L$ and a two-atom
  lattice~$U$ is obtained by removing the top of~$L$ and the bottom
  of~$U$, and identifying the coatoms of~$L$ with the atoms of~$U$.
  This operation creates one or two nonisomorphic lattices depending
  on the symmetry case.  Here the symmetry cases are analyzed, and a
  recurrence relation is presented that expresses the number of such
  vertical 2-sums in some family of interest, up to isomorphism.
  Nonisomorphic, vertically indecomposable modular and distributive
  lattices are counted and classified up to 35 and 60 elements
  respectively.  Asymptotically their numbers are shown to be at least
  $\Omega(2.3122^n)$ and $\Omega(1.7250^n)$, where $n$~is the number
  of elements.  The number of semimodular lattices is shown to grow
  faster than any exponential in~$n$.
\end{abstract}

\maketitle

\section{Introduction}
\label{sec:intro}

Let~$L$ and~$U$ be finite lattices.  Their \emph{vertical sum}
is~obtained by identifying the top of~$L$ with the bottom of~$U$.  If
$L$~has two~coatoms and $U$~has two~atoms, their \emph{vertical~2-sum}
is obtained by removing the top of~$L$ and the bottom of~$U$, and
identifying the coatoms of~$L$ with the atoms of~$U$.

The vertical sum leads to a simple and well-known recurrence relation.
A~lattice is a \emph{vi-lattice} (short for \emph{vertically
  indecomposable}) if it is not a vertical sum of two non-singleton
lattices.  If $f(n)$ and $f_\mathrm{vi}(n)$ are the numbers of nonisomorphic
$n$-element lattices and vi-lattices, respectively, then
\begin{equation}
  f(n) = \sum_{k=2}^n f_\mathrm{vi}(k) \, f(n-k+1), \qquad \text{for $n \ge 2$.}
  \label{eq:virule}
\end{equation}
We call this \emph{Cartesian counting} because each term expresses the
cardinality of a Cartesian product, namely, of the set of $k$-element
vi-lattices and the set of $(n\!-\!k\!+\!1)$-element lattices.  One
does not need to list the members of a Cartesian product to find its
cardinality.  Recurrence~\eqref{eq:virule} has been used in counting
small
lattices~\cite{erne2002,gebhardt2020,heitzig2002,jipsen2015,kohonen2018}
and in proving lower bounds~\cite{erne2002,kohonen2019}.

Many vi-lattices can be further decomposed as vertical 2-sums of
smaller lattices.  So let us pursue a~kind of Cartesian counting of
vi-lattices.  Now we must observe that from two given lattices, one
obtains two vertical 2-sums, because there are two ways to~match the
coatoms and the atoms.  Whether the results are isomorphic depends on
the symmetries of~$L$ and~$U$.

Our main result, Theorem~\ref{thm:main}, is a recurrence relation that
distinguishes the symmetry cases, and expresses the exact number of
nonisomorphic lattices obtainable as vertical 2-sums.  To apply the
recurrence, we need to classify and count the component lattices by
symmetry type.

The motivations of this study are threefold.  First, our recurrence
provides a new way of counting small lattices.  Only the component
lattices are generated explicitly; their vertical 2-sums are then
counted with the recurrence in the Cartesian fashion.  This is faster,
so we can count further.  We count modular and distributive
vi-lattices of at most 35 and 60 elements, respectively.  This also
provides a verification of previous countings (of at most 33 and 49
elements, respectively), because the method is different.

The second motivation is a more compact lattice listing.  A~full
listing of distributive vi-lattices of at most $60$ elements would
contain about $4.9 \times 10^{12}$ lattices.  We~can shrink the list
to less than $1/200$ of that size, to $2.3 \times 10^{10}$ lattices,
by leaving~out all vertical 2-sums.  A~smaller listing is more
practical to~store and to~study, and one can still recover the
left-out lattices by performing the vertical~2-sums.

The third motivation is in improving lower bounds.  A simple
recurrence for vertical 2-sums was derived in~\cite{kohonen2019}, but
it is only a loose lower bound as it does~not consider the symmetry
cases.  The new recurrence gives tighter bounds because of an extra
factor of~$2$ in the asymmetric cases.  It may not sound much, but the
factor compounds when vertical 2-sum is applied repeatedly.  Some
further improvement comes from counting small lattices larger than
before.  For nonisomorphic modular vi-lattices, we improve the lower
bound from $\Omega(2.1562^n)$~\cite{kohonen2019} to
$\Omega(2.3122^n)$.  For nonisomorphic distributive vi-lattices, we
improve from $\Omega(1.678^n)$~\cite{erne2002} to $\Omega(1.7250^n)$,
which is close to the empirical growth rate.

\section{Vertical 2-sum and symmetry}
\label{sec:symmetry}

In order to understand how the vertical 2-sum operates on lattices, we
classify them by the number and symmetry of their atoms and coatoms.
Our aim is in modular and distributive vi-lattices, but we state the
results more generally when convenient.  All lattices considered in
this work are finite.  If~$L$ is a~lattice, we write $a(L)$ and $c(L)$
for the numbers of its atoms and coatoms, and $0_L$~and~$1_L$ for its
bottom and top.

\begin{definition}
  Let $L$ and $U$ be disjoint lattices of length~3 or greater, $L$
  with two coatoms $c_1,c_2$ and $U$ with two atoms $a_1,a_2$.  Then
  their \emph{vertical 2-sums} are the two lattices obtained by
  removing $1_L$ and $0_U$, and identifying $(c_1,c_2)$ with either
  $(a_1,a_2)$ or $(a_2,a_1)$.  $L$ and $U$ are the \emph{summands} of
  the vertical 2-sum.
\end{definition}

Note that vertical 2-sums are indeed lattices, those of graded
lattices are graded, and those of vi-lattices are
vi-lattices~\cite{kohonen2019}.  We do not consider summands of
length~2 as~that would be essentially an identity operation.  If $S$
is a vertical 2-sum of $L$~and~$U$, then $|S| = |L|+|U|-4$.

\begin{definition}
  If a lattice has two coatoms [atoms], they are \emph{symmetric} if
  the lattice has an automorphism that swaps them, and \emph{fixed}
  otherwise.
\end{definition}

\begin{lemma}
  Let $L$ and $U$ be lattices with vertical 2-sums $S_1$~and~$S_2$.
  Then $S_1$~and~$S_2$ are nonisomorphic if and only if $L$~has fixed
  coatoms and $U$~has fixed atoms.
  \label{lemma:number}
\end{lemma}

\begin{proof}
  If $L$ has an automorphism that swaps the coatoms, then extending it
  with the identity mapping on $U$ yields an isomorphism $S_1 \to
  S_2$.  If $U$ has symmetric atoms, the case is similar.  Finally, if
  there is an isomorphism $S_1 \to S_2$, it must either fix the
  coatoms of~$L$ and swap the atoms of~$U$, or vice versa; but this is
  impossible if $L$ has fixed coatoms and $U$ has fixed atoms.
\end{proof}

From here on we confine our attention to graded vi-lattices.  We
divide them into three kinds as follows.

\begin{definition}
  If $L$ is a graded vi-lattice, then its $k$th \emph{level}, denoted
  $L_k$, is the set of elements that have rank~$k$.  A~\emph{neck} is
  a two-element level other than the atoms and the coatoms.  We say
  that $L$ is
  \begin{enumerate}
  \item a \emph{composition}, if it contains a neck;
  \item a \emph{piece}, if does not contain a neck, has rank~3 or
    greater, and at least one of $a(L)$ and $c(L)$ equals two;
  \item \emph{special} otherwise.
  \end{enumerate}
  \label{def:mainclass}
\end{definition}

A composition has necessarily at least 8~elements, and a piece has
at least~6.  All compositions ensue from pieces by repeated
application of the vertical 2-sum.  Specials are not vertical 2-sums,
but also cannot act as their summands, because they are too short
(rank two or smaller) or contain too many atoms and coatoms.

\begin{definition}
  A piece $L$ is:
  \begin{enumerate}
  \item a \emph{middle piece}, if $a(L)=c(L)=2$;
  \item a \emph{bottom piece}, if $a(L) \ge 3$ and $c(L) = 2$;
  \item a \emph{top piece}, if $a(L) = 2$ and $c(L) \ge 3$.
  \end{enumerate}
\end{definition}

A~middle piece can act as either summand of a vertical 2-sum.
A~bottom piece can act only as the lower summand, and a~top piece only
as the upper summand.

\begin{definition}
  A middle piece is of symmetry type:
  \begin{enumerate}
  \item MF, if its atoms and coatoms are fixed;
  \item MA, if its atoms are symmetric and coatoms are fixed;
  \item MC, if its atoms are fixed and coatoms are symmetric;
  \item MX, if it has an automorphism that swaps the atoms but fixes
    the coatoms, and another automorphism that swaps the coatoms but fixes
    the atoms;
  \item MH, if it is not MX, but has an automorphism that swaps both
    the atoms and the coatoms.
  \end{enumerate}
\end{definition}

\begin{definition}
  A bottom piece is of symmetry type:
  \begin{enumerate}
  \item BF, if its coatoms are fixed;
  \item BS, if its coatoms are symmetric.
  \end{enumerate}
\end{definition}

\begin{definition}
  A top piece is of symmetry type:
  \begin{enumerate}
  \item TF, if its atoms are fixed;
  \item TS, if its atoms are symmetric.
  \end{enumerate}
\end{definition}

\begin{figure}
  \subcaptionbox*{MF}
  {\includegraphics[height=0.23\textwidth]{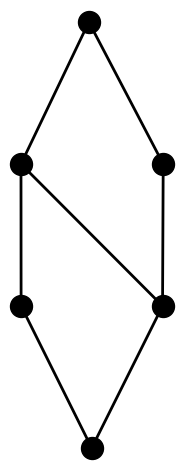}}
  \hspace{2em}
  \subcaptionbox*{MA}
  {\includegraphics[height=0.23\textwidth]{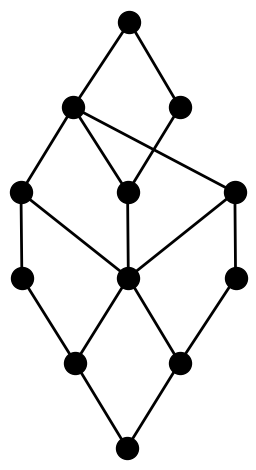}}
  \hspace{2em}
  \subcaptionbox*{MC}
  {\includegraphics[height=0.23\textwidth]{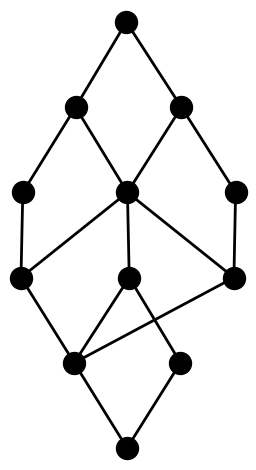}}
  \hspace{2em}
  \subcaptionbox*{MX}
  {\includegraphics[height=0.23\textwidth]{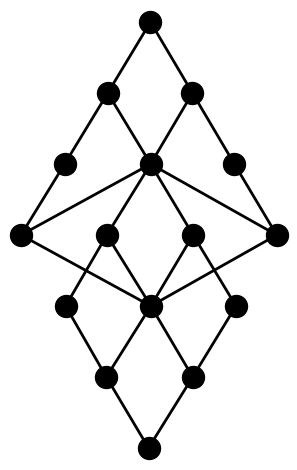}}
  \hspace{2em}
  \subcaptionbox*{MH}
  {\includegraphics[height=0.23\textwidth]{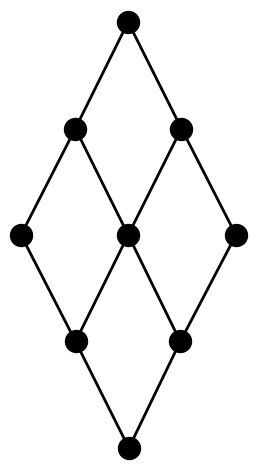}}
  \\[3ex]
  \subcaptionbox*{BF}
  {\includegraphics[height=0.23\textwidth]{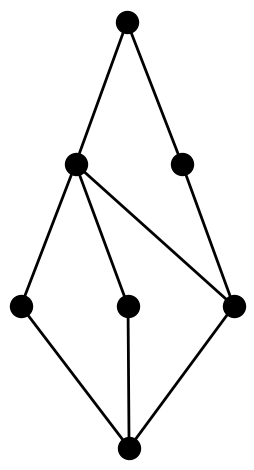}}
  \hspace{2em}
  \subcaptionbox*{BS}
  {\includegraphics[height=0.23\textwidth]{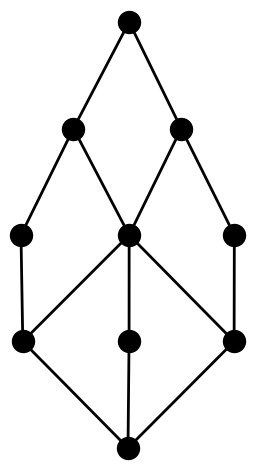}}
  \hspace{2em}
  \subcaptionbox*{TF}
  {\includegraphics[height=0.23\textwidth]{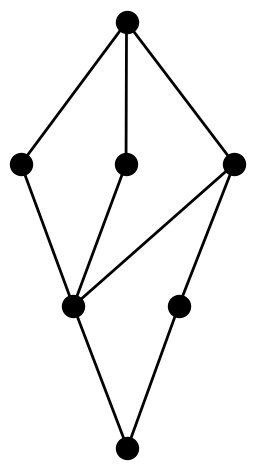}}
  \hspace{2em}
  \subcaptionbox*{TS}
  {\includegraphics[height=0.23\textwidth]{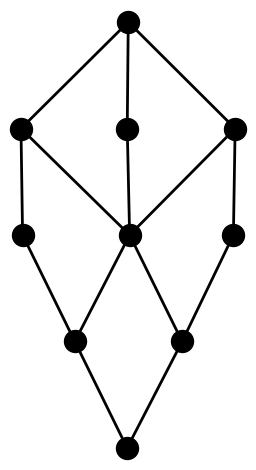}}
  \hspace{2em}
  \subcaptionbox*{special}
  {\includegraphics[height=0.23\textwidth]{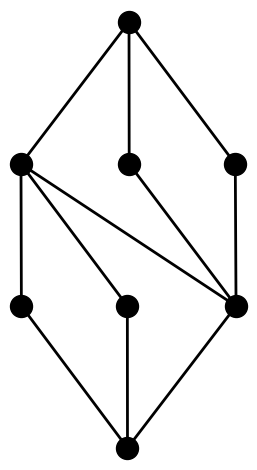}}
  \caption{Example pieces of each symmetry type (and a special).}
  \label{fig:pieces}
\end{figure}

\begin{figure}
  \includegraphics[height=0.32\textwidth]{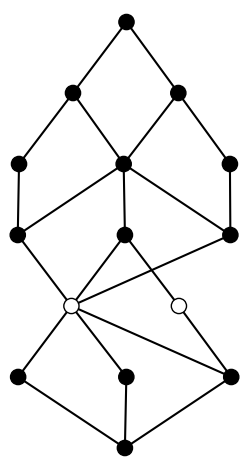}
  \hspace{4em}
  \includegraphics[height=0.32\textwidth]{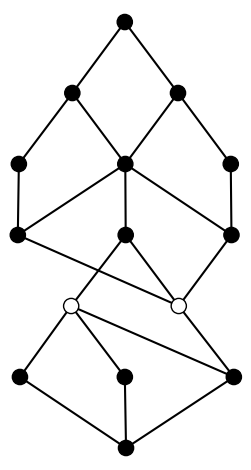}
  \caption{The two vertical 2-sums of a BF piece and an MC piece
    (elements common to the summands are shown as hollow circles).}
  \label{fig:gluing}
\end{figure}

\begin{definition}
  A composition is of symmetry type:
  \begin{enumerate}
  \item CF, if it has two coatoms and they are fixed;
  \item CS, if it has two coatoms and they are symmetric;
  \item CN (``composition-nonextensible''), if it has three or more coatoms.
  \end{enumerate}
\end{definition}

The symmetry types are illustrated in Figure~\ref{fig:pieces}.  If~we
take the BF~example as the lower summand, and the MC~example as the
upper summand, we obtain two nonisomorphic vertical 2-sums (both of
type CS) as shown in Figure~\ref{fig:gluing}.

Note that in~an MX piece atoms and coatoms can be swapped
independently, but in an MH piece only simultaneously; the shapes of
the letters X~and~H are meant as mnemonics for this.

We will form compositions ``bottom up'', adding new pieces over pieces
or smaller compositions.  The following lemma characterizes how many
and what kinds of compositions are formed in different cases.

\newcommand{\CF}{\mathrm{CF}}
\newcommand{\CS}{\mathrm{CS}}
\newcommand{\CN}{\mathrm{CN}}
\newcommand{\CFL}{\underline{\mathrm{CF}}}
\newcommand{\CSL}{\underline{\mathrm{CS}}}
\newcommand{\CNL}{\underline{\mathrm{CN}}}
\newcommand{\twice}{2\!\times\!}
\newcommand{\once}{1\!\times\!}
\newcommand{\XX}{\mathrm{XX}}
\newcommand{\LF}{\mathrm{LF}}
\newcommand{\LS}{\mathrm{LS}}
\newcommand{\MF}{\mathrm{MF}}
\newcommand{\MA}{\mathrm{MA}}
\newcommand{\MC}{\mathrm{MC}}
\newcommand{\MX}{\mathrm{MX}}
\newcommand{\MH}{\mathrm{MH}}
\newcommand{\BF}{\mathrm{BF}}
\newcommand{\BS}{\mathrm{BS}}
\newcommand{\TF}{\mathrm{TF}}
\newcommand{\TS}{\mathrm{TS}}

\begin{lemma}
  If $L$ is a piece or a composition, and $U$ is a piece, then the
  number and type of their nonisomorphic vertical 2-sums are as
  follows.

  \smallskip
  \setlength{\tabcolsep}{4pt}
  \nonslanted{  
    \renewcommand{\arraystretch}{1.5}
    \noindent\begin{tabular}{@{}lccccccc@{}}
    &&& \multicolumn{3}{c}{$U$ type} \\
    \cmidrule{2-8}
    $L$ type, any of & MF & MA & MC & MX & MH & TF & TS \\
    \midrule
    CF/BF/MF/MA     & $\twice\CF$ & $ \once\CF$ & $\twice\CS$ & $ \once\CS$ & $ \once\CF$ & $\twice\CN$ & $ \once\CN$ \\
    CS/BS/MC/MX/MH\hspace{1em} & $ \once\CF$ & $ \once\CF$ & $ \once\CS$ & $ \once\CS$ & $ \once\CS$ & $ \once\CN$ & $ \once\CN$ \\
    CN/TF/TS & none  & none & none & none & none & none & none \\
    \end{tabular}
  }
  \label{lemma:cases}
\end{lemma}

\begin{proof}
  Let us first prove the numbers.  We proceed by the rows.

  Row 1: $L$ has two fixed coatoms.  By~Lemma~\ref{lemma:number}, if
  $U$ has fixed atoms (MF, MC or TF), then there are two nonisomorphic
  vertical 2-sums; otherwise there is~one.
      
  Row 2: $L$ has two symmetric coatoms.  By~Lemma~\ref{lemma:number}
  there is one vertical 2-sum up to isomorphism.
      
  Row 3: $L$ has three or more coatoms, so no vertical 2-sums are formed.

  \smallskip
  Next we deduce the symmetry type of each vertical 2-sum~$S_i$
  ($i=1,2$).  We proceed by the columns.

  If $U$ is MF or MA, it has fixed coatoms.  Then $S_i$ cannot have an
  automorphism that swaps the coatoms of $S_i$, because its
  restriction to $U$ would be an automorphism that swaps the coatoms
  of~$U$.  Thus $S_i$ has fixed coatoms.
    
  If $U$ is MC or MX, it has an automorphism that swaps its coatoms
  and fixes its atoms; extending with the identity mapping in~$L$
  gives, in each $S_i$, an automorphism that swaps the coatoms.  Thus
  $S_i$ has symmetric coatoms.
  
  If $U$ is MH, then an automorphism of $S_i$ that swaps the coatoms
  also swaps the atoms of $U$, which are also the coatoms of $L$.
  Thus $S_i$ has symmetric coatoms if and only if $L$ has symmetric
  coatoms.

  If $U$ is TF or TS, then $S_i$ has three or more coatoms and is of
  type~CN.
\end{proof}

\section{Cartesian counting in a lattice family}

In this section we present our main result, a recurrence relation that
counts all nonisomorphic compositions in some desired lattice family,
provided that the family has suitable form.  We also give examples of
such families.

\begin{definition}
  A family $\F$ of graded vi-lattices is (vertically)
  \emph{2-summable} if the following conditions hold:
  \begin{enumerate}[(C1)]
  \item If $L,U \in \F$ and $S$ is one of their vertical 2-sums, then
    $S \in \F$.
  \item If $S \in \F$ is a vertical 2-sum of $L$ and~$U$, then
    $L,U \in F$.
  \end{enumerate}
  \label{def:summable}
\end{definition}

The first condition ensures that vertical 2-sums stay in the family,
and the second ensures that all compositions in $\F$ are indeed
obtained as vertical 2-sums of smaller lattices in~$\F$.

\begin{theorem}
  Let $\F$ be a 2-summable family, and let $\XX_n$ denote the number
  of nonisomorphic $n$-element lattices in~$\F$ having symmetry
  type~XX.  Then for $n < 8$ we have $\CF_n = \CS_n = \CN_n = 0$, and
  for $n \ge 8$ the following recurrences hold:
  \begin{align*}
    \CF_n &= \sum_{j=6}^{n-2} \Big( \LF_j \cdot (2\cdot\MF_k + \MA_k + \MH_k)
            \;+\; \LS_j \cdot (\MF_k + \MA_k) \Big) \\
    \CS_n &= \sum_{j=6}^{n-2} \Big( \LF_j \cdot (2\cdot\MC_k + \MX_k)
            \;+\; \LS_j \cdot (\MC_k + \MX_k + \MH_k) \Big) \\
    \CN_n &= \sum_{j=6}^{n-2} \Big( \LF_j \cdot (2\cdot\TF_k + \TS_k)
            \;+\; \LS_j \cdot (\TF_k + \TS_k) \\
    \intertext{where $k=n-j+4$, and}
    \LF_j &= \CF_j + \BF_j + \MF_j + \MA_j \\
    \LS_j &= \CS_j + \BS_j + \MC_j + \MX_j + \MH_j.
  \end{align*}
  \label{thm:main}
\end{theorem}

\begin{proof}
  For $n<8$ the numbers are zero, because a composition cannot have
  fewer than $8$~elements.
  
  Let then $n \ge 8$ and consider an $n$-element CF-type composition
  $S \in \F$.  There is exactly one way of expressing $S$ as a
  vertical 2-sum of two lattices $L,U$ such that $U$~is a piece.  This
  $U$ contains the elements of $S$ above and including its
  highest-ranked neck, plus an augmented bottom element.  By
  condition~(C2) we have $L,U \in \F$.  Furthermore, because
  $|U| \ge 6$ and and $|L|+|U|-4=n$, it follows that $|L| \le n-2$.

  We also observe that different nonisomorphic choices of $L$ and $U$,
  where $U$ is a piece, lead to nonisomorphic results.  To be more
  precise: If $L \ncong L'$ or $U \ncong U'$, and $U$ and $U'$ are
  pieces, then the vertical 2-sums of $L$ and $U$ are not isomorphic
  to the vertical 2-sums of $L'$ and $U'$.
  
  All nonisomorphic $n$-element CF-type compositions in $\F$ can be
  counted by considering (for all $j=6,\ldots,n-2$) first the choices
  of a $j$-element lower summand $L \in \F$, and then the choices of
  an upper summand $U \in \F$ such that $U$~is a~piece with $k=n-j+4$
  elements, subject to the requirement that the resulting vertical
  2-sums are of type~CF.

  Now $\LF_j$ is the number of nonisomorphic lower summands that have
  fixed coatoms.  For each such lower summand, by collecting the
  CF-type results from the first row of the table in
  Lemma~\ref{lemma:cases}, we get $2\cdot\MF_k+\MA_k+\MH_k$
  nonisomorphic vertical 2-sums, which are in~$\F$ by condition~(C1).

  Similarly, $\LS_j$ is the number of nonisomorphic lower summands
  that have symmetric coatoms.  For each such lower summand, by
  collecting the CF-type results from the second row of the table in
  Lemma~\ref{lemma:cases}, we get $\MF_k+\MA_k$ nonisomorphic vertical
  2-sums, which are in~$\F$ by condition~(C1).

  Adding up the cases we obtain the stated expression for $\CF_n$.
  The expressions for $\CS_n$ and $\CN_n$ follow in the same manner.
\end{proof}

Not all families of graded vi-lattices are 2-summable.  For a~simple
example, finite graded rank-four vi-lattices fail both conditions (C1)
and~(C2).  Interestingly, finite geometric lattices are 2-summable but
in a vacuous way.

\begin{theorem}
  The only finite geometric lattice that has a two-element level is
  $M_2$.
  \label{thm:geometric-vacuous}
\end{theorem}

\begin{proof}
  Let $L$~be a finite geometric lattice that has a two-element
  level~$\{a,b\}$.  Because $L$ is atomistic, neither $a$ or $b$ has
  any join-irreducible covers, thus $a$ and $b$ are covered by exactly
  one element~$c$.  Further, because $L$~is necessarily vertically
  indecomposable, it follows that $c=1_L$, and $a,b$ are the coatoms.

  The numbers of atoms and coatoms in a finite geometric lattice of
  rank~$r$ are known as the Whitney numbers $W_1$ and $W_{r-1}$.
  It~is known that $W_1 \le W_{r-1}$; see e.g.~Dowling and
  Wilson~\cite{dowling1975}.  This implies that our~$L$ has two atoms.
  But we have shown that $L$~cannot have two-element levels other than
  the coatoms; thus the atoms are the coatoms, and $L = M_2$.
\end{proof}

In other words, Theorem~\ref{thm:geometric-vacuous} says that all
finite geometric lattices are special; there are no pieces and no
compositions, so no use for the vertical 2-sum.  Modular and
distributive lattices will be more interesting for our purposes.  We
first prove an auxiliary result by elementary means.

\begin{lemma}
  In the Hasse diagram of a finite semimodular lattice, the subgraph
  induced by two consecutive levels is connected.
  \label{lemma:connected}
\end{lemma}

\begin{proof}
  Let $L$ be a finite semimodular lattice, $L_k$ its $k$th level, and
  $H_k$ the subgraph induced by $L_k$ and $L_{k+1}$.  We use induction
  on~$k$.  Clearly $H_0$ is connected.  Assume then that $H_{k-1}$ is
  connected.  If $L_k$ is a singleton, then obviously $H_k$ is
  connected.  Otherwise, let $(U,V)$ be any partition of $L_k$ into
  two nonempty subsets.  Because $H_{k-1}$ is connected, there is an
  element in $L_{k-1}$ that is covered by some two elements $u \in U$
  and $v \in V$.  Then by semimodularity $u$ and $v$ are covered by
  some $w \in L_{k+1}$, so there is a path from $U$ to~$V$ in~$H_k$.
  Finally, from every element in $L_{k+1}$ there is an edge to $L_k$.
  Thus $H_k$ is connected.
\end{proof}

Note that Lemma~\ref{lemma:connected} also follows from previously
known, more general results:  Bj\"{o}rner~\cite{bjorner1980} proved
that finite semimodular lattices are lexicographically shellable, and
Collins~\cite{collins1992} proved that graded lexicographically
shellable lattices are rank-connected (i.e.~the subgraph induced by
two consecutive levels is connected).

\begin{lemma}
  The family of finite semimodular vi-lattices is 2-summable.
  \label{lemma:semimodular-summable}
\end{lemma}

\begin{proof}
  We use subscripted symbols $\wedge_L$, $\vee_L$ and $\prec_L$ to
  denote meet, join and covered-by in lattice~$L$.
  
  \begin{enumerate}[(C1)]
  \item Let $L$ and $U$ be finite semimodular vi-lattices, $S$ their
    vertical 2-sum, and $N$ the two identified elements of $L$
    and~$U$.  Clearly $S$ is a vi-lattice.  We show that $S$~is
    semimodular by using Birkhoff's
    condition~\cite[p.~331]{gratzer2011}.  Let $a,b \in S$ such that
    they cover $a \wedge_S b$.  Then $a$ and $b$ have the same rank,
    and are either below~$N$, in~$N$ or above~$N$.  If $a,b$ are
    below~$N$, then because $L$~is semimodular,
    $a,b \prec_S a \vee_L b = a \vee_S b$.  If $a,b$ are in or
    above~$N$, then because $U$~is semimodular,
    $a,b \prec_S a \vee_U b = a \vee_S b$.

  \item Let $S$ be a finite semimodular vi-lattice that is a vertical
    2-sum of $L$ and $U$.  Clearly $L$ and~$U$ are vi-lattices.  We
    show that they are semimodular, again using Birkhoff's condition.

    First let $a,b \in L$ such that they cover $a \wedge_L b$.  If
    $a,b$ are not the coatoms of~$L$, then $a \vee_L b = a \vee_S b$.
    If $a,b$ are the coatoms of~$L$, then $a \vee_L b = 1_L$.  In both
    cases $a,b \prec_L a \vee_L b$.  Thus $L$~is semimodular.

    Let then $a,b \in U$ such that they cover $a \wedge_U b$.  If
    $a,b$ are not the atoms of~$U$, then
    $a \wedge_U b = a \wedge_S b$, and because $S$ is semimodular,
    $a,b \prec_U a \vee_U b = a \vee_S b$.  If $a,b$ are the atoms
    of~$U$, then they are a neck of~$S$.  Because $S$~is semimodular,
    it follows from Lemma~\ref{lemma:connected} that $a,b$ have a
    common upper cover~$c$ in~$S$.  Then also $a,b \prec_U c$.  Thus
    $U$~is semimodular. \qedhere
  \end{enumerate}
\end{proof}
   
\begin{lemma}
  The family of finite modular vi-lattices is 2-summable.
  \label{lemma:modular-summable}
\end{lemma}

\begin{proof}
  Follows from Lemma~\ref{lemma:semimodular-summable} by duality.
\end{proof}

\begin{lemma}
  The family of finite distributive vi-lattices is 2-summable.
  \label{lemma:distributive-summable}
\end{lemma}

\begin{proof}
  We recall that a finite modular lattice is distributive if and only
  if it does not contain a~cover-preserving
  diamond~\cite[p.~109]{gratzer2011}, that is, five distinct elements
  $o,a,b,c,i$ such that $o \prec a,b,c \prec i$.
  \begin{enumerate}[(C1)]
  \item Let $L$ and $U$ be finite distributive vi-lattices and $S$
    their vertical 2-sum.  By~Lemma~\ref{lemma:modular-summable} $S$
    is modular.  Since $L$ and $U$ do not contain a cover-preserving
    diamond, the only possibility for~$S$ to contain one would be with
    $o \in L$ and $i \in U$, but the neck consisting of the two
    identified elements of $L$ and~$U$ cannot contain three distinct
    elements $a,b,c$.  Thus $S$~is distributive.
    
  \item Let $S$ be a finite distributive vi-lattice that is a vertical
    2-sum of $L$ and~$U$.  By Lemma~\ref{lemma:modular-summable} $L$
    and $U$ are modular vi-lattices.  Since $S$ does not contain
    a~cover-preserving diamond, the only possibility for $L$ to
    contain one would be with $i = 1_L$, but this is impossible
    because $L$ has only two coatoms.  Thus $L$~is distributive.  The
    case of~$U$ is similar. \qedhere
  \end{enumerate}
\end{proof}

\section{Computations}

\subsection{Method of classifying a lattice}
\label{subsec:classifying}

Given a graded vi-lattice represented by its covering graph, a short
piece of program code classifies the lattice into the types described
in Section~\ref{sec:symmetry}.  Calculating lattice length, counting
atoms and coatoms, and finding possible necks is straightforward.  For
analyzing the symmetry type we use the Nauty
library, version 2.7r1~\cite{nauty,mckay2014}.

Nauty returns the automorphism group of a given directed graph as a
list of generators $(\gamma_1, \ldots, \gamma_k)$.  To classify a
bottom piece we check if any generator swaps the coatoms; in that case
the coatoms are symmetric, otherwise fixed.  With a top piece we check
if any generator swaps the~atoms.

To classify a middle piece some more cases are required.  For each
generator, we check if it:
\begin{enumerate}[(A)]
  \item swaps the atoms and fixes the coatoms; or
  \item swaps both the atoms and the coatoms; or
  \item swaps the coatoms and fixes the coatoms.
\end{enumerate}
Generators that touch neither atoms nor coatoms are ignored.  Then:
\begin{enumerate}[(1)]
\item If there are no generators of types A/B/C, the piece is MF.
\item If there are generators of type A, but none of B/C, the piece is MA.
\item If there are generators of type C, but none of A/B, the piece is MC.
\item If there are generators of type B, but none of A/C, the piece is MH.
\item If there are generators of at least two of the types A/B/C, the piece is MX.
\end{enumerate}
It is easily seen that this procedure produces the correct
classification.  Note that for an MX piece, Nauty does not necessarily
return generators of types A~and~C.  It can instead return, for
example, a~generator $\gamma_i$ of type~A, and a~generator $\gamma_j$
of type~B.  But then $\gamma_i \circ \gamma_j$ is an automorphism that
swaps the coatoms and fixes the atoms, and then we know that the piece
is indeed~MX.

\subsection{Modular lattices}

Modular vi-lattices were previously generated and counted up to 30
elements in~\cite{kohonen2018}, and up to 33 elements in unpublished
work~\cite{a006981}.  That was done with a program that starts from
length-two seed lattices, and then adds new levels of elements
recursively.  The program lists exactly one representative lattice
from each isomorphism class.

We use here essentially the same lattice-generating program, modified
so that it skips all compositions, and generates only the pieces and
the specials.  The modification is simply that two-element levels are
not allowed between coatom and atom levels, because such a level would
form a neck.

With this program, all modular pieces and specials of $n \le 35$
elements were generated (up to isomorphism), and classified as
described in \S\ref{subsec:classifying}.  Because modular
vi-lattices are 2-summable (Lemma~\ref{lemma:modular-summable}), the
numbers of modular compositions of $n \le 35$ elements are then
calculated using Theorem~\ref{thm:main}.

The results of the exact counting are shown in
Table~\ref{table:modular}.  Rows MF--TF and ``special'' are from
direct counting with the lattice-generating program.  Rows CF--CN are
calculated with the recurrence in Theorem~\ref{thm:main}.  Row
``vi-latt.''  contains the numbers of all modular vi-lattices: this is
the sum of specials, pieces and compositions.  Finally, row ``all''
has the numbers of all modular lattices (including vertical sums of
vi-lattices), calculated with the vertical sum
recurrence~\eqref{eq:virule}.

An~exponential lower bound is derived as follows.  Using
Theorem~\ref{thm:main} with the known numbers of modular pieces of up
to 35~elements, and plugging in zeros for larger pieces (whose numbers
we do not know), we obtain lower bounds on $\CF_n$, $\CS_n$ and
$\CN_n$ for $n$ arbitrarily large.  We observe that the growth ratios
(from $n$ to $n+1$) of all three lower bounds settle a~little above
2.3122 for $n$~large enough.  To obtain rigorous lower bounds, we
choose a~convenient starting point $n=50$, convenient constant
coefficients in front, and apply induction.

\begin{theorem}
  The numbers of nonisomorphic modular compositions of types CF, CS
  and CN have the following lower bounds, when $n \ge 50$:
  \begin{align*}
    \CF_n &\ge 0.002910 \times 2.3122^n \\
    \CS_n &\ge 0.000035 \times 2.3122^n \\
    \CN_n &\ge 0.002470 \times 2.3122^n
  \end{align*}
  \label{thm:modular-bound}
\end{theorem}

\begin{proof}
  For $50 \le n \le 85$ the claim follows by direct calculation with
  Theorem~\ref{thm:main}, using the numbers of pieces from
  Table~\ref{table:modular}, and zeros when the number of pieces is
  not known.

  For $n > 85$ the claim follows by induction on~$n$.  Let $n > 85$ be
  arbitrary, and assume that the claimed lower bounds hold on $\CF_m$,
  $\CS_m$ and $\CN_m$ when $n-35 \le m \le n-1$.  Then applying
  Theorem~\ref{thm:main} gives the claimed lower bounds on $\CF_n$,
  $\CS_n$ and $\CN_n$, which completes the induction.
\end{proof}

\begin{corollary}
  There are at least $0.005415 \times 2.3122^n$ nonisomorphic modular
  vi-lattices of~$n$~elements when $n \ge 50$.
\end{corollary}

\begin{proof}
  Add up the three lower bounds from Theorem~\ref{thm:modular-bound}.
\end{proof}

\begin{corollary}
  There are at least $0.02 \times 2.3713^n$ nonisomorphic modular
  lattices of~$n$~elements when $n \ge 100$.
\end{corollary}

\begin{proof}
  For $100 \le n \le 400$ the claim follows by direct calculation
  with recurrence~\eqref{eq:virule}, using as input the lower bounds on
  vi-lattices of up to $400$ elements computed using
  Theorem~\ref{thm:modular-bound}.

  For the induction step, let $n > 400$ be arbitrary, and
  assume that the claimed lower bound holds for the previous $300$
  values.  Applying~\eqref{eq:virule} then gives the claimed lower
  bound for the number of $n$-element modular lattices.  This completes
  the induction.
\end{proof}
  
\subsection{Distributive lattices}
\label{subsec:distributive}

Distributive vi-lattices were previously counted up~to 49 elements by
Ern\'{e} et al.~\cite{erne2002,a006982}.  To~count $n$-element
distributive lattices, they actually generated posets that have
$n$~antichains; these are in a bijective correspondence with the
distributive lattices.

Our approach is more direct.  We generate the distributive lattices
directly, using the same program that we used for modular lattices,
with some modifications.  The first modification is a condition that
ensures that we generate only the distributive lattices.  Since the
original program generates modular lattices, we only need to check
that whenever a new element is created, it does not create a
cover-preserving diamond~\cite[p.~109]{gratzer2011}.  This ensures
that we generate the distributive lattices but no~others.

We also employ an important optimization that cuts short search
branches that cannot lead to distributive lattices.  Our
lattice-generating program (see~\cite{kohonen2018} for more details)
builds lattices levelwise, top down, until the number of elements
reaches a preset maximum.  When creating a new level, it adds new
elements in decreasing order of updegree.  The last step on each level
is thus to create meet-irreducible elements.  In the original
algorithm, this step can create a large number of meet-irreducible
elements, limited only by the maximum lattice size.  But in a
distributive lattice we can limit their number as follows.  We~recall
(see Corollary~112 in~\cite{gratzer2011}) that the number of
meet-irreducible elements in a distributive lattice equals the lattice
length.  As~we build a lattice, we keep track of the meet-irreducible
elements created so far, and at~each level we compute an updated upper
bound~$R$ on the lattice length (based on the current length and the
budget of remaining elements).  The number of meet-irreducible
elements, including the ones already created, is then limited to be at
most~$R$.

With this program, all distributive pieces and specials of $n \le 60$
elements were generated (up to isomorphism), and classified with the
method described in \S\ref{subsec:classifying}.  Compositions
were then counted using Theorem~\ref{thm:main}.  The results are shown
in Table~\ref{table:distributive}.

An~exponential lower bound is derived in the same way as with
modulars.  Using Theorem~\ref{thm:main} with the known numbers of
distributive pieces of up to 60~elements, and plugging in zeros for
larger pieces, we obtain lower bounds on $\CF_n$, $\CS_n$ and $\CN_n$
whose growth ratios settle a~little above 1.7250 for $n$~large enough.
To obtain rigorous lower bounds, we choose a~convenient starting point
$n=100$, convenient constant coefficients in front, and apply
induction.

\begin{theorem}
  The numbers of nonisomorphic distributive compositions of types CF,
  CS and CN have the following lower bounds, when $n \ge 100$:
  \begin{align*}
    \CF_n &\ge 0.010600 \times 1.7250^n \\
    \CS_n &\ge 0.000092 \times 1.7250^n \\
    \CN_n &\ge 0.001950 \times 1.7250^n
  \end{align*}
  \label{thm:distributive-bound}
\end{theorem}

\begin{proof}
  For $100 \le n \le 161$ the claim follows by direct calculation with
  Theorem~\ref{thm:main}, using the numbers of pieces from
  Table~\ref{table:distributive}, and zeros when the number of pieces
  is not known.

  For $n > 161$ the claim follows by induction on~$n$.  Let $n > 161$ be
  arbitrary, and assume that the claimed lower bounds hold on $\CF_m$,
  $\CS_m$ and $\CN_m$ when $n-61 \le m \le n-1$.  Then applying
  Theorem~\ref{thm:main} gives the claimed lower bounds on $\CF_n$,
  $\CS_n$ and $\CN_n$, which completes the induction.
\end{proof}

\begin{corollary}
  There are at least $0.012642 \times 1.7250^n$ nonisomorphic
  distributive vi-lattices of~$n$~elements when $n \ge 100$.
\end{corollary}

\begin{proof}
  Add up the three lower bounds from
  Theorem~\ref{thm:distributive-bound}.
\end{proof}

\begin{corollary}
  There are at least $0.088 \times 1.8433^n$ nonisomorphic
  distributive lattices of~$n$~elements when $n \ge 100$.
\end{corollary}

\begin{proof}
  For $100 \le n \le 400$ the claim follows by direct calculation with
  recurrence~\eqref{eq:virule}, using as input the lower bounds on
  vi-lattices of up to $400$ elements computed using
  Theorem~\ref{thm:distributive-bound}.

  For the induction step, let $n > 400$ be arbitrary, and assume that
  the claimed lower bound holds for the previous $300$ values.
  Applying~\eqref{eq:virule} then gives the claimed lower bound for the
  number of $n$-element distributive lattices.  This completes the
  induction.
\end{proof}

\subsection{Semimodular lattices}

Although semimodular vi-lattices are 2-summable, vertical 2-sum is not
very useful with them.  For example, of the $1\;753\;185\;150$
semimodular vi-lattices of $25$~elements, only about $23\%$ are
compositions.  This is in stark contrast with modular and distributive
lattices.  Basically this is because semimodular lattices are short
and wide (cf.~Figures 4--5 in~\cite{kohonen2018}).  For this reason we
do not include tables of semimodular lattices here, but such tables
can be easily computed using the accompanying program code.

We could apply the same techniques as above to obtain an exponential
lower bound.  But an asymptotically stronger lower bound is obtained
by constructing semimodular lattices from Steiner triple systems.
A~\emph{Steiner triple system} is a set of $k$~elements
(\emph{points}) and a collection of their 3-sets (\emph{triples}),
such that each pair of distinct points occurs in exactly one triple.
By counting the pairs it is easily seen that the number of triples
must be $k(k-1)/6$.  It is known that Steiner triple systems on
$k$~points exist if and only if $k \equiv 1 \text{ or } 3 \pmod 6$;
such values of $k$ are called \emph{admissible}.

Given a Steiner triple system on $k \ge 7$ points, if we take the
points as atoms, the triples as coatoms, and augment a top and a
bottom, we obtain a rank-three semimodular vi-lattice, because each
pair of atoms is covered by exactly one coatom.  The lattice has
$k+k(k-1)/6+2$ elements.

\begin{theorem}
  For any $n\ge 100$, the number of nonisomorphic semimodular
  rank-three vi-lattices containing $n$~elements is at least
  \[
  \big( 0.3286 \; n^{1/8} \big)^n.
  \]
  \label{thm:steinerbound}
\end{theorem}

\begin{proof}
  Let $n \ge 100$ be given.  Choose the largest admissible~$k$ such
  that
  \[
  k+k(k-1)/6+2 \le n.
  \]
  Then $k \ge 21$, and because admissible values are at most 4~units
  apart, we have
  \begin{equation}
    n < (k+4) + (k+4)(k+3)/6+2 < k^2/3.
    \label{eq:steinerrounding}
  \end{equation}
  Let $N$ be the number of nonisomorphic Steiner triple systems
  on $k$~points.  By~Wilson's Theorem~2~\cite{wilson1974}, we have
  \[
  N \ge (e^{-5}k)^{k^2/12},
  \]
  and using the bound $k > \sqrt{3n}$
  from~\eqref{eq:steinerrounding} we obtain
  \[
  N \ge \big(e^{-5} \cdot \sqrt{3n} \big)^{n/4}
  = \big(e^{-5/4} \cdot 3^{1/8} \cdot n^{1/8} \big)^n
  \ge \big(0.3286 \; n^{1/8} \big)^n.
  \]

  From these $N$ Steiner triple systems on $k$~points, we can form $N$
  nonisomorphic semimodular rank-three vi-lattices that have $n' =
  k+k(k-1)/6+2$ elements, and by our choice of~$k$, we have $n' \le
  n$.  To each lattice, add $n-n'$ extra coatoms covering an
  arbitrarily chosen atom of the highest updegree.  This operation
  makes the lattices have exactly $n$~elements, and preserves
  semimodularity and nonisomorphism, so the claim follows.
\end{proof}

The lower bound in Theorem~\ref{thm:steinerbound} is very loose (it
does not even exceed~$1$ until $n \approx 7000$), and is presented in
simple terms just to demonstrate the asymptotic behaviour, that the
number of semimodular vi-lattices grows faster than any exponential
in~$n$.  The bound might be improved in several ways, for example, by
using Keevash's recent improvement on Wilson's lower
bound~\cite{keevash2018}.

\subsection{Notes on computation}
The main computational load was in generating the~pieces and specials.
For~the largest sizes this was parallelized by running the
lattice-generating program until a~predefined number of elements had
been added.  The search state at those points was saved to a~file, and
the remaining work was divided among worker processes.  For modular
lattices of $33$, $34$ and $35$~elements, this computation took $8.9$,
$23.7$ and $63.1$ cpu-core-days on Intel Xeon Gold 6230 processors
(nominally~2.1~GHz).  The~time grows roughly $2.66$-fold
as~$n$~grows by one.  For~distributive lattices of $58$, $59$ and
$60$~elements, the computation took $6.3$, $10.4$ and $17.2$
cpu-core-days, growing roughly $1.66$-fold each time.

The optimization in \S\ref{subsec:distributive} that made it possible
to count distributive lattices up to $60$ elements is just one
example, out of many, where the speed of a combinatorial search is
greatly affected by a simple, innocent-looking bounding condition.
Here its implementation takes about a~dozen lines of code (see the
accompanying program code for details).  But already at $30$ elements,
it speeds up the generation of distributive pieces and specials from
153~seconds to 0.4~seconds; and the savings ratio keeps improving as
the lattices grow.

Unfortunately, it is not always easy to come up with conditions that
actually have a great impact, and are also fast enough to compute
during the search.  Given the large existing theory of the structure
of distributive and modular lattices, it is conceivable that our
lattice-generating program could still be much improved by imposing
some other, so far untried, bounding conditions.

\subsection{Partial verification}

We describe here some of the methods that were used to partially
verify the correctness of the computational results.

The pieces and specials were generated and classified twice, on
different computers.  The counts and the actual lattice listings were
verified to be identical by comparing their MD5 checksums.  This would
help against transient hardware and operational errors, but not
against systematic errors in the program the.

The number of MA~pieces equals the number of MC~pieces in each column
of Tables~\ref{table:modular} and~\ref{table:distributive}.  This is
as it should, because such pieces are duals of each other.  The same
holds between BF~and~TF pieces, and between BS~and~TS pieces.

We also performed a more thorough duality check.  The \emph{rank
  sequence} of a~graded lattice is the sequence of its level sizes
from bottom to top.  The rank sequences of a lattice and its dual are
reverses of each other.  In~Figure~\ref{fig:pieces} the BF example
piece has rank sequence $(1,3,2,1)$, and its dual, the TF example
piece has $(1,2,3,1)$.  We counted the occurrences of each combination
of symmetry type and rank sequence in pieces and specials, and
verified that the numbers match between the dual pairs (MF--MF,
MA--MC, MX--MX, MH--MH, BF--TF, BS--TF, and special--special).  For
example, among all $60$-element distributive MX pieces, there are
$2\;137$ whose rank sequence is
$(1,2,3,4,6,5,4,4,4,4,4,4,3,3,3,3,2,1)$, and $2\;137$ whose sequence
is its reverse.  If any of the lattice listings were accidentally
truncated or corrupted, this would likely be detected as a mismatch.
An error in the lattice-generating logic or in the classification code
would also have a good chance of causing a mismatch.

As a consistency check of the Cartesian counting logic, we directly
generated, classified and counted all distributive vi-lattices of
50~elements \emph{including} compositions.  In each class, the count
thus obtained matches the calculated count in
Table~\ref{table:distributive}.  Generating these 50-element
lattices took $12.6$ cpu-core-days, about 70~times longer than with
compositions excluded, which also demonstrates the benefits of
Cartesian counting.

Our totals on rows ``vi-latt.'' and ``all'' agree with previously
published numbers of modular lattices to 33
elements~\cite{kohonen2018,a006981} and distributive lattices to 49
elements~\cite{erne2002} (except at $n=1$, because we count the
singleton as a vi-lattice).  The previous countings did not employ the
vertical 2-sum.


\section{Concluding remarks}

One of our stated goals was to create more compact lattice listings by
leaving out all compositions (vertical 2-sums of smaller lattices).
As~seen in Tables~\ref{table:modular} and~\ref{table:distributive},
this was more successful with distributive lattices than with
modulars.  Compositions make up $79\%$ of the modular vi-lattices of
$35$~elements, and $99.6\%$ of the distributive vi-lattices of
$60$~elements.

The observed growths of the numbers of modular and distributive
vi-lattices, to~the extent that they are now known, are illustrated in
Figure~\ref{fig:ratioplots}.  In modular vi-lattices our lower bound
$\Omega(2.3122^n)$ seems loose; the observed ratios keep increasing,
hinting perhaps of~a (very slightly) superexponential growth.  We note
that no~exponential upper bound is currently known on the number of
modular lattices.  In distributive vi-lattices the observed ratios
seem to be converging, and our lower bound $\Omega(1.7250^n)$ seems
pretty good.

\begin{figure}[tb]
  \includegraphics[width=\textwidth]{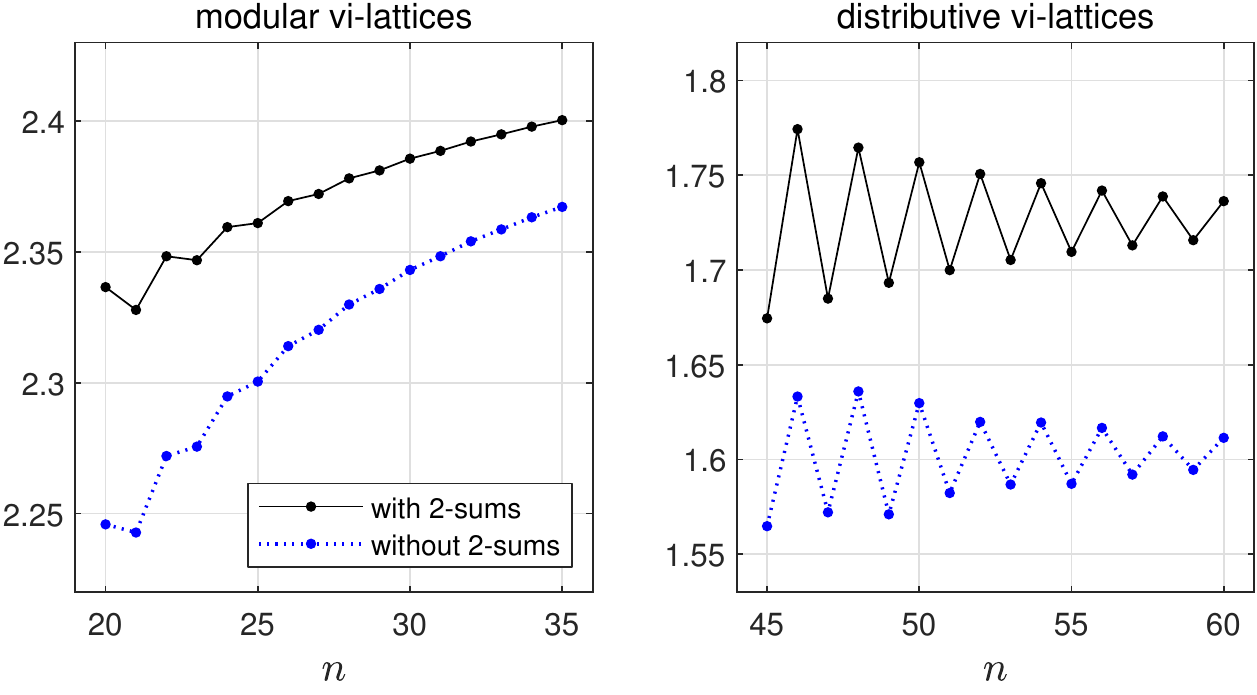}
  \caption{Ratio between the numbers of nonisomorphic vi-lattices of
    $(n\!-\!1)$ and $n$ elements, with and without vertical 2-sums.}
  \label{fig:ratioplots}
\end{figure}

Ern\'{e} et al.\ have shown an upper bound of $O(2.33^n)$ on
nonisomorphic distributive vi-lattices~\cite{erne2002}.  Our improved
lower bound $\Omega(2.3122^n)$ on nonisomorphic modular vi-lattices is
still not enough to separate the growth rates of these two families.
To~close the gap there are different options.  We could count modular
pieces further.  Empirically, adding one element increases the base in
our lower bound by $0.0048$ (but the increase diminishes as $n$
grows).  Counting the pieces up~to $40$ elements would probably raise
the lower bound above $\Omega(2.33^n)$.  But this would take about
$10\;000$ cpu-core-days with the current lattice-generating program,
and was deemed not worth the effort.  Improving the algorithm or the
lower bound techniques might be a better idea.  Another option is to
improve the upper bound on distributive vi-lattices.  Indeed, Ern\'{e}
et~al.\ note that ``with more effort'' it might be improved
considerably, at least to~$2.28^n$.  Combined with our lower bound,
this would suffice to separate the growth rates.

Although our lower bounds are based on large computations, we must
point~out that proper analysis of symmetry is the key to good lower
bounds.  Indeed, using all our data on distributive lattices ($n\le
60$), if we ignore the symmetry cases (and lose the factors of~2 in
Theorem~\ref{thm:main}), we only obtain a bound of $\Omega(1.6213^n)$
for distributive vi-lattices.  In contrast, using just the
distributive middle pieces of $n \le 21$ elements (a truly modest
collection of 134 lattices), our symmetry-distinguishing method
already gives $\Omega(1.6818^n)$.

Accompanying program code is available in Bitbucket~\cite{bitbucket}.
This includes C~programs to generate the pieces and the specials, to
classify and count them by symmetry type, and to perform the Cartesian
counting.  Also included is SageMath code for verifying the
exponential lower bounds.  The~lattice listings (pieces and specials
only) were stored in digraph6 format and compressed with~xz.  The
compressed listings take about 167~GB of disk space, and will be
available in~\cite{eudat}.


\begin{table}
  \caption{Number of modular lattice up to isomorphism.}
  \label{table:modular}
  \small
\setlength{\tabcolsep}{4.5pt}
\begin{tabular}{@{}lrrrrrrrrrrrrrrrr@{}}
\toprule
& $n$ \\
type &            1 &            2 &            3 &            4 &            5 &            6 &            7 &            8 &            9 &           10 &           11 &           12 &           13 &           14 &           15 &           16 \\
\midrule
\addlinespace
  MF &            0 &            0 &            0 &            0 &            0 &            1 &            0 &            0 &            2 &            3 &            3 &           13 &           24 &           48 &          105 &          242 \\  MA &            0 &            0 &            0 &            0 &            0 &            0 &            0 &            0 &            0 &            0 &            0 &            1 &            1 &            2 &            5 &            7 \\  MC &            0 &            0 &            0 &            0 &            0 &            0 &            0 &            0 &            0 &            0 &            0 &            1 &            1 &            2 &            5 &            7 \\  MX &            0 &            0 &            0 &            0 &            0 &            0 &            0 &            0 &            0 &            0 &            0 &            0 &            0 &            0 &            0 &            1 \\  MH &            0 &            0 &            0 &            0 &            0 &            0 &            0 &            0 &            1 &            0 &            1 &            0 &            1 &            0 &            2 &            2 \\  BF &            0 &            0 &            0 &            0 &            0 &            0 &            1 &            1 &            1 &            4 &            6 &           11 &           25 &           56 &          113 &          257 \\  BS &            0 &            0 &            0 &            0 &            0 &            0 &            0 &            0 &            0 &            1 &            1 &            2 &            4 &            5 &            9 &           15 \\  TF &            0 &            0 &            0 &            0 &            0 &            0 &            1 &            1 &            1 &            4 &            6 &           11 &           25 &           56 &          113 &          257 \\  TS &            0 &            0 &            0 &            0 &            0 &            0 &            0 &            0 &            0 &            1 &            1 &            2 &            4 &            5 &            9 &           15 \\\addlinespace
  CF &            0 &            0 &            0 &            0 &            0 &            0 &            0 &            2 &            2 &            6 &           16 &           38 &           80 &          208 &          464 &         1115 \\
  CS &            0 &            0 &            0 &            0 &            0 &            0 &            0 &            0 &            0 &            0 &            0 &            0 &            0 &            3 &            5 &           15 \\
  CN &            0 &            0 &            0 &            0 &            0 &            0 &            0 &            0 &            2 &            4 &           10 &           28 &           66 &          154 &          375 &          884 \\
\midrule
special  &            1 &            1 &            0 &            1 &            1 &            1 &            1 &            3 &            3 &            5 &           10 &           20 &           35 &           75 &          151 &          317 \\pieces &            0 &            0 &            0 &            0 &            0 &            1 &            2 &            2 &            5 &           13 &           18 &           41 &           85 &          174 &          361 &          803 \\
compos. &            0 &            0 &            0 &            0 &            0 &            0 &            0 &            2 &            4 &           10 &           26 &           66 &          146 &          365 &          844 &         2014 \\
\midrule
vi-latt. &            1 &            1 &            0 &            1 &            1 &            2 &            3 &            7 &           12 &           28 &           54 &          127 &          266 &          614 &         1356 &         3134 \\
all &            1 &            1 &            1 &            2 &            4 &            8 &           16 &           34 &           72 &          157 &          343 &          766 &         1718 &         3899 &         8898 &        20475 \\
\bottomrule\end{tabular}  
\\[1cm]
\setlength{\tabcolsep}{5pt}
\begin{tabular}{@{}lrrrrrrrr@{}}
\toprule
& $n$ \\
type &           17 &           18 &           19 &           20 &           21 &           22 &           23 &           24 \\
\midrule
\addlinespace
  MF &          518 &         1185 &         2664 &         6092 &        13849 &        31932 &        73458 &       170112 \\  MA &           15 &           28 &           61 &          122 &          270 &          570 &         1259 &         2729 \\  MC &           15 &           28 &           61 &          122 &          270 &          570 &         1259 &         2729 \\  MX &            1 &            4 &            5 &           11 &           18 &           35 &           63 &          124 \\  MH &            5 &            8 &           11 &           19 &           22 &           43 &           51 &          105 \\  BF &          557 &         1250 &         2763 &         6267 &        14125 &        32225 &        73561 &       169304 \\  BS &           30 &           52 &          109 &          207 &          422 &          835 &         1721 &         3544 \\  TF &          557 &         1250 &         2763 &         6267 &        14125 &        32225 &        73561 &       169304 \\  TS &           30 &           52 &          109 &          207 &          422 &          835 &         1721 &         3544 \\\addlinespace
  CF &         2580 &         6156 &        14382 &        34236 &        80703 &       192141 &       455548 &      1086269 \\
  CS &           35 &           84 &          191 &          457 &         1054 &         2482 &         5795 &        13601 \\
  CN &         2091 &         4959 &        11736 &        27832 &        66009 &       156845 &       372956 &       888193 \\
\midrule
special  &          657 &         1426 &         3074 &         6783 &        15006 &        33707 &        75944 &       172893 \\pieces &         1728 &         3857 &         8546 &        19314 &        43523 &        99270 &       226654 &       521495 \\
compos. &         4706 &        11199 &        26309 &        62525 &       147766 &       351468 &       834299 &      1988063 \\
\midrule
vi-latt. &         7091 &        16482 &        37929 &        88622 &       206295 &       484445 &      1136897 &      2682451 \\
all &        47321 &       110024 &       256791 &       601991 &      1415768 &      3340847 &      7904700 &     18752943 \\
\bottomrule\end{tabular}
\end{table}

\begin{table}
  \caption*{{\sc Table 1} (continued). Number of modular lattice up to isomorphism.}
  \small
\setlength{\tabcolsep}{5pt}
\begin{tabular}{@{}lrrrrrr@{}}
\toprule
& $n$ \\
type &           25 &           26 &           27 &           28 &           29 &           30 \\
\midrule
\addlinespace
  MF &       394356 &       918597 &      2142885 &      5016593 &     11766661 &     27673169 \\  MA &         6054 &        13395 &        29981 &        67308 &       152290 &       345897 \\  MC &         6054 &        13395 &        29981 &        67308 &       152290 &       345897 \\  MX &          239 &          474 &          945 &         1911 &         3917 &         8094 \\  MH &          148 &          290 &          454 &          826 &         1359 &         2352 \\  BF &       390258 &       904769 &      2102583 &      4905597 &     11472236 &     26908706 \\  BS &         7475 &        15902 &        34379 &        75030 &       165752 &       369140 \\  TF &       390258 &       904769 &      2102583 &      4905597 &     11472236 &     26908706 \\  TS &         7475 &        15902 &        34379 &        75030 &       165752 &       369140 \\\addlinespace
  CF &      2586652 &      6179943 &     14763845 &     35347971 &     84670699 &    203133686 \\
  CS &        31931 &        75120 &       176999 &       417863 &       988002 &      2340245 \\
  CN &      2117276 &      5054559 &     12078748 &     28902161 &     69228582 &    166012187 \\
\midrule
special  &       395073 &       908830 &      2098043 &      4866320 &     11320574 &     26427788 \\pieces &      1202317 &      2787493 &      6478170 &     15115200 &     35352493 &     82931101 \\
compos. &      4735859 &     11309622 &     27019592 &     64667995 &    154887283 &    371486118 \\
\midrule
vi-latt. &      6333249 &     15005945 &     35595805 &     84649515 &    201560350 &    480845007 \\
all &     44588803 &    106247120 &    253644319 &    606603025 &   1453029516 &   3485707007 \\
\bottomrule\end{tabular}
  \\[1cm]
\setlength{\tabcolsep}{5pt}
\begin{tabular}{@{}lrrrrr@{}}
\toprule
& $n$ \\
type &           31 &           32 &           33 &           34 &           35 \\
\midrule
\addlinespace
  MF &     65203834 &    153963391 &    364151886 &    862779754 &   2047145114 \\  MA &       790496 &      1813615 &      4180886 &      9673363 &     22467366 \\  MC &       790496 &      1813615 &      4180886 &      9673363 &     22467366 \\  MX &        16975 &        35876 &        76749 &       165615 &       360878 \\  MH &         3958 &         6696 &        11466 &        19465 &        33807 \\  BF &     63245392 &    148991342 &    351620380 &    831365583 &   1968780807 \\  BS &       829576 &      1877307 &      4277558 &      9800078 &     22571155 \\  TF &     63245392 &    148991342 &    351620380 &    831365583 &   1968780807 \\  TS &       829576 &      1877307 &      4277558 &      9800078 &     22571155 \\\addlinespace
  CF &    487682310 &   1172237243 &   2819860668 &   6789965627 &  16361898245 \\
  CS &      5551716 &     13191092 &     31388574 &     74798062 &    178482514 \\
  CN &    398494238 &    957517799 &   2302844911 &   5543373958 &  13354884177 \\
\midrule
special  &     61853133 &    145160950 &    341431589 &    804878006 &   1901058538 \\pieces &    194955695 &    459370491 &   1084397749 &   2564642882 &   6075178455 \\
compos. &    891728264 &   2142946134 &   5154094153 &  12408137647 &  29895264936 \\
\midrule
vi-latt. &   1148537092 &   2747477575 &   6579923491 &  15777658535 &  37871501929 \\
all &   8373273835 &  20139498217 &  48496079939 & 116905715114 & 282098869730 \\
\bottomrule\end{tabular}
\end{table}


\begin{table}
  \caption{Number of distributive lattices up to isomorphism.}
  \label{table:distributive}
  \small
\setlength{\tabcolsep}{4.5pt}
\begin{tabular}{@{}lrrrrrrrrrrrrrrrrr@{}}
\toprule
& $n$ \\
type &            1 &            2 &            3 &            4 &            5 &            6 &            7 &            8 &            9 &           10 &           11 &           12 &           13 &           14 &           15 &           16 &           17 \\
\midrule
\addlinespace
  MF &            0 &            0 &            0 &            0 &            0 &            1 &            0 &            0 &            0 &            0 &            0 &            3 &            0 &            0 &            4 &            5 &            4 \\  MA &            0 &            0 &            0 &            0 &            0 &            0 &            0 &            0 &            0 &            0 &            0 &            0 &            0 &            0 &            1 &            0 &            0 \\  MC &            0 &            0 &            0 &            0 &            0 &            0 &            0 &            0 &            0 &            0 &            0 &            0 &            0 &            0 &            1 &            0 &            0 \\  MX &            0 &            0 &            0 &            0 &            0 &            0 &            0 &            0 &            0 &            0 &            0 &            0 &            0 &            0 &            0 &            0 &            0 \\  MH &            0 &            0 &            0 &            0 &            0 &            0 &            0 &            0 &            1 &            0 &            0 &            0 &            0 &            0 &            0 &            1 &            0 \\  BF &            0 &            0 &            0 &            0 &            0 &            0 &            0 &            0 &            0 &            1 &            0 &            0 &            0 &            2 &            1 &            3 &            1 \\  BS &            0 &            0 &            0 &            0 &            0 &            0 &            0 &            0 &            0 &            0 &            0 &            0 &            1 &            0 &            0 &            0 &            1 \\  TF &            0 &            0 &            0 &            0 &            0 &            0 &            0 &            0 &            0 &            1 &            0 &            0 &            0 &            2 &            1 &            3 &            1 \\  TS &            0 &            0 &            0 &            0 &            0 &            0 &            0 &            0 &            0 &            0 &            0 &            0 &            1 &            0 &            0 &            0 &            1 \\\addlinespace
  CF &            0 &            0 &            0 &            0 &            0 &            0 &            0 &            2 &            0 &            4 &            2 &           10 &            6 &           32 &           18 &           83 &           74 \\
  CS &            0 &            0 &            0 &            0 &            0 &            0 &            0 &            0 &            0 &            0 &            0 &            0 &            0 &            1 &            0 &            0 &            2 \\
  CN &            0 &            0 &            0 &            0 &            0 &            0 &            0 &            0 &            0 &            0 &            0 &            2 &            0 &            4 &            2 &           14 &            8 \\
\midrule
special  &            1 &            1 &            0 &            1 &            0 &            0 &            0 &            1 &            0 &            0 &            0 &            1 &            0 &            1 &            0 &            3 &            1 \\pieces &            0 &            0 &            0 &            0 &            0 &            1 &            0 &            0 &            1 &            2 &            0 &            3 &            2 &            4 &            8 &           12 &            8 \\
compos. &            0 &            0 &            0 &            0 &            0 &            0 &            0 &            2 &            0 &            4 &            2 &           12 &            6 &           37 &           20 &           97 &           84 \\
\midrule
vi-latt. &            1 &            1 &            0 &            1 &            0 &            1 &            0 &            3 &            1 &            6 &            2 &           16 &            8 &           42 &           28 &          112 &           93 \\
all &            1 &            1 &            1 &            2 &            3 &            5 &            8 &           15 &           26 &           47 &           82 &          151 &          269 &          494 &          891 &         1639 &         2978 \\
\bottomrule\end{tabular}\\[1cm]
\setlength{\tabcolsep}{5pt}
\begin{tabular}{@{}lrrrrrrrrr@{}}
\toprule
& $n$ \\
type &           18 &           19 &           20 &           21 &           22 &           23 &           24 &           25 &           26 \\
\midrule
\addlinespace
  MF &           16 &           10 &           29 &           49 &           63 &           94 &          213 &          219 &          459 \\  MA &            0 &            1 &            1 &            2 &            1 &            3 &            3 &           10 &            8 \\  MC &            0 &            1 &            1 &            2 &            1 &            3 &            3 &           10 &            8 \\  MX &            0 &            0 &            0 &            0 &            0 &            0 &            1 &            0 &            1 \\  MH &            1 &            0 &            0 &            0 &            1 &            1 &            0 &            3 &            2 \\  BF &            8 &            7 &           16 &           16 &           40 &           38 &          102 &          116 &          229 \\  BS &            0 &            1 &            1 &            2 &            1 &            5 &            2 &            8 &            4 \\  TF &            8 &            7 &           16 &           16 &           40 &           38 &          102 &          116 &          229 \\  TS &            0 &            1 &            1 &            2 &            1 &            5 &            2 &            8 &            4 \\\addlinespace
  CF &          230 &          233 &          672 &          726 &         1928 &         2342 &         5516 &         7280 &        16178 \\
  CS &            1 &            5 &            2 &           14 &            9 &           37 &           27 &           99 &           95 \\
  CN &           41 &           27 &          120 &          104 &          343 &          347 &         1005 &         1119 &         2953 \\
\midrule
special  &            6 &            2 &           10 &            6 &           26 &           18 &           56 &           48 &          131 \\pieces &           33 &           28 &           65 &           89 &          148 &          187 &          428 &          490 &          944 \\
compos. &          272 &          265 &          794 &          844 &         2280 &         2726 &         6548 &         8498 &        19226 \\
\midrule
vi-latt. &          311 &          295 &          869 &          939 &         2454 &         2931 &         7032 &         9036 &        20301 \\
all &         5483 &        10006 &        18428 &        33749 &        62162 &       114083 &       210189 &       386292 &       711811 \\
\bottomrule\end{tabular}
\end{table}

\begin{table}
  \caption*{{\sc Table 2} (continued). Number of distributive lattices up to isomorphism.}
  \small
\setlength{\tabcolsep}{5pt}
\begin{tabular}{@{}lrrrrrrr@{}}
\toprule
& $n$ \\
type &           27 &           28 &           29 &           30 &           31 &           32 &           33 \\
\midrule
\addlinespace
  MF &          726 &         1099 &         1691 &         3112 &         4176 &         7573 &        11728 \\  MA &           15 &           14 &           50 &           47 &           87 &          121 &          227 \\  MC &           15 &           14 &           50 &           47 &           87 &          121 &          227 \\  MX &            0 &            2 &            0 &            4 &            3 &            5 &            4 \\  MH &            1 &            2 &            2 &            8 &            2 &           10 &            6 \\  BF &          303 &          596 &          749 &         1513 &         2033 &         3647 &         5316 \\  BS &           16 &           11 &           32 &           25 &           59 &           62 &          151 \\  TF &          303 &          596 &          749 &         1513 &         2033 &         3647 &         5316 \\  TS &           16 &           11 &           32 &           25 &           59 &           62 &          151 \\\addlinespace
  CF &        22302 &        47348 &        68582 &       138752 &       208961 &       409676 &       632745 \\
  CS &          281 &          301 &          789 &          926 &         2307 &         2865 &         6611 \\
  CN &         3594 &         8607 &        11348 &        25363 &        35198 &        74935 &       108658 \\
\midrule
special  &          129 &          328 &          339 &          769 &          914 &         1913 &         2371 \\pieces &         1395 &         2345 &         3355 &         6294 &         8539 &        15248 &        23126 \\
compos. &        26177 &        56256 &        80719 &       165041 &       246466 &       487476 &       748014 \\
\midrule
vi-latt. &        27701 &        58929 &        84413 &       172104 &       255919 &       504637 &       773511 \\
all &      1309475 &      2413144 &      4442221 &      8186962 &     15077454 &     27789108 &     51193086 \\
\bottomrule\end{tabular}
  \\[1cm]
\setlength{\tabcolsep}{5pt}
\begin{tabular}{@{}lrrrrrr@{}}
\toprule
& $n$ \\
type &           34 &           35 &           36 &           37 &           38 &           39 \\
\midrule
\addlinespace
  MF &        18593 &        29332 &        49894 &        73906 &       125464 &       196346 \\  MA &          279 &          584 &          732 &         1333 &         1963 &         3362 \\  MC &          279 &          584 &          732 &         1333 &         1963 &         3362 \\  MX &           18 &           15 &           20 &           22 &           71 &           68 \\  MH &           16 &            9 &           16 &           26 &           31 &           40 \\  BF &         9431 &        13450 &        24024 &        35267 &        60195 &        91542 \\  BS &          147 &          317 &          369 &          755 &          927 &         1833 \\  TF &         9431 &        13450 &        24024 &        35267 &        60195 &        91542 \\  TS &          147 &          317 &          369 &          755 &          927 &         1833 \\\addlinespace
  CF &      1211099 &      1914417 &      3583636 &      5772993 &     10632469 &     17361550 \\
  CS &         8863 &        19257 &        27094 &        56362 &        82534 &       165301 \\
  CN &       221306 &       333256 &       655975 &      1014990 &      1947706 &      3081099 \\
\midrule
special  &         4783 &         6192 &        11888 &        16279 &        29902 &        42083 \\pieces &        38341 &        58058 &       100180 &       148664 &       251736 &       389928 \\
compos. &      1441268 &      2266930 &      4266705 &      6844345 &     12662709 &     20607950 \\
\midrule
vi-latt. &      1484392 &      2331180 &      4378773 &      7009288 &     12944347 &     21039961 \\
all &     94357143 &    173859936 &    320462062 &    590555664 &   1088548290 &   2006193418 \\
\bottomrule\end{tabular}
\end{table}

\begin{table}
  \caption*{{\sc Table 2} (continued). Number of distributive lattices up to isomorphism.}
  \small
\setlength{\tabcolsep}{5pt}
\begin{tabular}{@{}lrrrrr@{}}
\toprule
& $n$ \\
type &           40 &           41 &           42 &           43 &           44 \\
\midrule
\addlinespace
  MF &       316251 &       501232 &       824706 &      1277065 &      2104201 \\  MA &         4763 &         8706 &        12369 &        21206 &        32381 \\  MC &         4763 &         8706 &        12369 &        21206 &        32381 \\  MX &          137 &          132 &          314 &          389 &          739 \\  MH &           39 &           72 &           82 &          100 &          131 \\  BF &       154501 &       234005 &       395667 &       606372 &      1005172 \\  BS &         2286 &         4424 &         5957 &        10857 &        15117 \\  TF &       154501 &       234005 &       395667 &       606372 &      1005172 \\  TS &         2286 &         4424 &         5957 &        10857 &        15117 \\\addlinespace
  CF &     31576370 &     52152822 &     93836823 &    156430892 &    279134095 \\
  CS &       250864 &       486894 &       758311 &      1439247 &      2287530 \\
  CN &      5786183 &      9323041 &     17211623 &     28123776 &     51240833 \\
\midrule
special  &        75946 &       109160 &       191940 &       283583 &       488243 \\pieces &       639527 &       995706 &      1653088 &      2554424 &      4210411 \\
compos. &     37613417 &     61962757 &    111806757 &    185993915 &    332662458 \\
\midrule
vi-latt. &     38328890 &     63067623 &    113651785 &    188831922 &    337361112 \\
all &   3697997558 &   6815841849 &  12563729268 &  23157428823 &  42686759863 \\
\bottomrule\end{tabular}
  \\[1cm]
\setlength{\tabcolsep}{5pt}
\begin{tabular}{@{}lrrrr@{}}
\toprule
& $n$ \\
type &           45 &           46 &           47 &           48 \\
\midrule
\addlinespace
  MF &      3324318 &      5359557 &      8530466 &     13845649 \\  MA &        54061 &        81642 &       139008 &       210288 \\  MC &        54061 &        81642 &       139008 &       210288 \\  MX &          858 &         1710 &         2298 &         4073 \\  MH &          157 &          263 &          252 &          415 \\  BF &      1566178 &      2579549 &      4019531 &      6615167 \\  BS &        27265 &        38591 &        68176 &        99857 \\  TF &      1566178 &      2579549 &      4019531 &      6615167 \\  TS &        27265 &        38591 &        68176 &        99857 \\\addlinespace
  CF &    468610361 &    830767951 &   1402696804 &   2473422299 \\
  CS &      4259645 &      6888356 &     12629975 &     20697992 \\
  CN &     84668721 &    152604200 &    254469592 &    454703002 \\
\midrule
special  &       731917 &      1246418 &      1890209 &      3178981 \\pieces &      6620341 &     10761094 &     16986446 &     27700761 \\
compos. &    557538727 &    990260507 &   1669796371 &   2948823293 \\
\midrule
vi-latt. &    564890985 &   1002268019 &   1688673026 &   2979703035 \\
all &  78682454720 & 145038561665 & 267348052028 & 492815778109 \\
\bottomrule\end{tabular}
\end{table}

\begin{table}
  \caption*{{\sc Table 2} (continued). Number of distributive lattices up to isomorphism.}
  \small
\setlength{\tabcolsep}{5pt}
\begin{tabular}{@{}lrrrr@{}}
\toprule
& $n$ \\
type &           49 &           50 &           51 &           52 \\
\midrule
\addlinespace
  MF &     21848698 &     35484402 &     56423044 &     90846703 \\  MA &       349946 &       545640 &       894103 &      1392365 \\  MC &       349946 &       545640 &       894103 &      1392365 \\  MX &         5652 &        10074 &        14075 &        24970 \\  MH &          440 &          697 &          770 &         1034 \\  BF &     10365640 &     16917992 &     26705669 &     43421020 \\  BS &       171474 &       256172 &       436542 &       657562 \\  TF &     10365640 &     16917992 &     26705669 &     43421020 \\  TS &       171474 &       256172 &       436542 &       657562 \\\addlinespace
  CF &   4195545640 &   7366918781 &  12541052681 &  21947209314 \\
  CS &     37503756 &     62106875 &    111453001 &    186162978 \\
  CN &    763638695 &   1355285715 &   2289013709 &   4040233169 \\
\midrule
special  &      4883596 &      8125938 &     12584095 &     20810796 \\pieces &     43628910 &     70934781 &    112510517 &    181814601 \\
compos. &   4996688091 &   8784311371 &  14941519391 &  26173605461 \\
\midrule
vi-latt. &   5045200597 &   8863372090 &  15066614003 &  26376230858 \\
all & 908414736485 & 1674530991462 & 3086717505436 & 5689930182502 \\
\bottomrule\end{tabular}
  \\[1cm]
\setlength{\tabcolsep}{5pt}
\begin{tabular}{@{}lrrrr@{}}
\toprule
& $n$ \\
type &           53 &           54 &           55 &           56 \\
\midrule
\addlinespace
  MF &    144993779 &    233835914 &    372140014 &    600341635 \\  MA &      2298377 &      3582354 &      5856047 &      9243133 \\  MC &      2298377 &      3582354 &      5856047 &      9243133 \\  MX &        35929 &        61578 &        91665 &       155593 \\  MH &         1401 &         1762 &         2365 &         2920 \\  BF &     68622251 &    111467609 &    176619879 &    285832100 \\  BS &      1108440 &      1696359 &      2822398 &      4362272 \\  TF &     68622251 &    111467609 &    176619879 &    285832100 \\  TS &      1108440 &      1696359 &      2822398 &      4362272 \\\addlinespace
  CF &  37469608053 &  65395110178 & 111905017483 & 194884875094 \\
  CS &    331472539 &    557422364 &    986467033 &   1667764379 \\
  CN &   6854924656 &  12046207362 &  20511977357 &  35920244327 \\
\midrule
special  &     32424737 &     53285185 &     83549296 &    136565579 \\pieces &    289089245 &    467391898 &    742830692 &   1199375158 \\
compos. &  44656005248 &  77998739904 & 133403461873 & 232472883800 \\
\midrule
vi-latt. &  44977519230 &  78519416987 & 134229841861 & 233808824537 \\
all & 10488501786986 & 19334113091637 & 35639590512519 & 65696773057331 \\
\bottomrule\end{tabular}
\end{table}

\begin{table}
  \caption*{{\sc Table 2} (continued). Number of distributive lattices up to isomorphism.}
  \small
\setlength{\tabcolsep}{5pt}
\begin{tabular}{@{}lrrrr@{}}
\toprule
& $n$ \\
type &           57 &           58 &           59 &           60 \\
\midrule
\addlinespace
  MF &    958148836 &   1540236160 &   2462775718 &   3959945640 \\  MA &     15005164 &     23705048 &     38546064 &     60946820 \\  MC &     15005164 &     23705048 &     38546064 &     60946820 \\  MX &       231975 &       392608 &       596474 &       990499 \\  MH &         3980 &         5148 &         6470 &         8675 \\  BF &    454454916 &    733959291 &   1168085737 &   1885053587 \\  BS &      7218116 &     11214722 &     18454173 &     28883114 \\  TF &    454454916 &    733959291 &   1168085737 &   1885053587 \\  TS &      7218116 &     11214722 &     18454173 &     28883114 \\\addlinespace
  CF & 334097123844 & 580839511384 & 997199063829 & 1731270488614 \\
  CS &   2936919383 &   4986729668 &   8746804291 &  14902405273 \\
  CN &  61338845389 & 107114456244 & 183330850349 & 319426226966 \\
\midrule
special  &    215048026 &    350313997 &    553415624 &    898644768 \\pieces &   1911741183 &   3078392038 &   4913550610 &   7910711856 \\
compos. & 398372888616 & 692940697296 & 1189276718469 & 2065599120853 \\
\midrule
vi-latt. & 400499677825 & 696369403331 & 1194743684703 & 2074408477477 \\
all & 121102696325898 & 223236665889804 & 411506035223499 & 758556959660012 \\
\bottomrule\end{tabular}
\end{table}

\subsection*{Acknowledgments}

Computational resources were provided by CSC -- IT Center for Science
and by the Aalto Science-IT project.


\bibliographystyle{plain}
\bibliography{refs}

\end{document}